\documentclass[12pt]{article}
\usepackage{amssymb,amsmath}
\usepackage{amsthm}
\usepackage{cases}
\usepackage{amsfonts}
\usepackage{graphicx}
\usepackage{cite,color,xcolor}
\usepackage[left=2.2cm,right=2.2cm,top=2.2cm,bottom=2.2cm]{geometry}
\usepackage[colorlinks,citecolor=blue,urlcolor=blue]{hyperref}
\usepackage[utf8]{inputenc}

\newtheorem{theorem}{Theorem}[section]
\newtheorem{corollary}{Corollary}[section]
\newtheorem{lemma}{Lemma}[section]
\newtheorem{proposition}{Proposition}[section]

\newtheorem{definition}{Definition}[section]
\newtheorem{remark}{Remark}[section]

\usepackage{txfonts}
\newcommand{\bal}{\begin{align}}
\newcommand{\bbal}{\begin{align*}}
\newcommand{\beq}{\begin{equation}}
\newcommand{\eeq}{\end{equation}}
\newcommand{\bca}{\begin{cases}}
\newcommand{\eca}{\end{cases}}
\def\div{\mathord{{\rm div}}}
\newcommand{\pa}{\partial}
\newcommand{\fr}{\frac}

\newcommand{\De}{\dot{\Delta}}

\newcommand{\cd}{\cdot}

\newcommand{\dd}{\mathrm{d}}

\newcommand{\LL}{\tilde{L}}
\newcommand{\R}{\mathbb{R}}

\newcommand{\les}{\lesssim}

\newcommand{\f}{\left}
\newcommand{\g}{\right}

\begin{document}
\bibliographystyle{plain}
\title{On the ill-posedness for the Navier--Stokes equations in the weakest Besov spaces}

\author{Yanghai Yu\footnote{E-mail:  yuyanghai214@sina.com; lijinlu@gnnu.edu.cn}\; and Jinlu Li  \\
\small $^1$ School of Mathematics and Statistics, Anhui Normal University, Wuhu 241002, China\\
\small $^2$ School of Mathematics and Computer Sciences, Gannan Normal University, Ganzhou 341000, China}

\date{\today}
\maketitle\noindent{\hrulefill}

{\bf Abstract:}  It is proved in \cite{IO21} that the Cauchy problem for the full compressible Navier--Stokes equations
of the ideal gas is ill-posed in $\dot{B}_{p, q}^{2 / p}(\mathbb{R}^2) \times \dot{B}_{p, q}^{2 / p-1}(\mathbb{R}^2)  \times \dot{B}_{p, q}^{2 / p-2}(\mathbb{R}^2) $ with $1\leq p\leq \infty$ and $1\leq q<\infty$. In this paper, we aim to solve the end-point case left in \cite{IO21} and prove that the Cauchy problem  is ill-posed in $\dot{B}_{p, \infty}^{d / p}(\mathbb{R}^d) \times \dot{B}_{p, \infty}^{d / p-1}(\mathbb{R}^d) \times \dot{B}_{p, \infty}^{d / p-2}(\mathbb{R}^d)$ with $1\leq p\leq\infty$ by constructing a sequence of initial data which shows that the solution map is discontinuous at zero. As a by-product, we demonstrate that the incompressible Navier--Stokes equations is also ill-posed in $\dot{B}_{p,\infty}^{d/p-1}(\mathbb{R}^d)$, which is an interesting open problem in itself.


\vskip0mm\noindent{\hrulefill}

\section{Introduction}
The full compressible Navier--Stokes equations in $\mathbb{R}^d$ with $d\geq2$ read as follows:
\begin{align}\label{00}\tag{fCNS}
\begin{cases}
\partial_t \tilde{\rho}+\operatorname{div}(\tilde{\rho} u)=0, \\
\partial_t(\tilde{\rho} u)+\operatorname{div}(\tilde{\rho} u \otimes u)=\operatorname{div} \tau, \\
\partial_t\left(\tilde{\rho} u\left(e+\frac{|u|^2}{2}\right)\right)+\operatorname{div}\left(\tilde{\rho} u\left(e+\frac{|u|^2}{2}\right)\right)=\operatorname{div}(\tau \cdot u+\kappa \nabla \theta),
\end{cases}
\end{align}
where the unknown functions $\tilde{\rho}(t, x), u(t, x)=\left(u_1(t, x), \cdots, u_d(t, x)\right), e(t, x)$ denote the density, velocity of the fluid and the intern energy per unit mass respectively, $\kappa>0$ is the thermal conduction parameter, and $\theta(t, x)$ is the temperature. The internal stress tensor $\tau$ is given by
$$
\tau=2 \mu \mathrm{D}(u)+(\lambda \operatorname{div} u-P) \mathrm{Id},
$$
where $\mu$ and $\lambda$ are the Lam\'{e} constants satisfying $\mu>0$ and $2 \mu+\lambda>0$. The strain tensor $\mathrm{D}(u)=(\nabla u+\nabla^{\mathsf{T}}u)/2$  is the symmetric part of the velocity gradient $\nabla u$, whose $(i,j)$-th component is given by $(\mathrm{D}u)_{ij}=(\partial_iu_j+\partial_ju_i)/2$ with $1\leq i,j\leq d$.

For the ideal gas, $e=c_V \theta$ and $P=\tilde{\rho} R \theta$ for some constants $c_V>0, R>0$. We note that the other physical constants except for $\mu, \lambda, \kappa$ are taken by 1 for simplicity, namely, $c_V =R=1$. In such case, the system \eqref{00} can be rewritten as
\begin{align}\label{1}
\begin{cases}
\partial_t \tilde{\rho}+\operatorname{div}(\tilde{\rho} u)=0, \\
\partial_t(\tilde{\rho} u)+\operatorname{div}(\tilde{\rho} u \otimes u)+\nabla(\tilde{\rho} \theta)=\mu \Delta u+(\mu+\lambda) \nabla \operatorname{div} u, \\
\partial_t(\tilde{\rho} \theta)+\operatorname{div}(\tilde{\rho} \theta u)+\tilde{\rho} \theta \operatorname{div} u-\kappa \Delta \theta=2 \mu|\mathrm{D}(u)|^2+\lambda|\operatorname{div} u|^2.
\end{cases}
\end{align}
Here $|A|^2$ denotes the trace of $A A^{\top}$ for a matrix $A$ and its transpose $A^{\top}$.
When the pressure depends only on the density, we get the baratropic Navier--Stokes equations
\begin{align}\label{cns}
\begin{cases}
\partial_t \tilde{\rho}+\operatorname{div}(\tilde{\rho} u)=0, \\
\partial_t(\tilde{\rho} u)+\operatorname{div}(\tilde{\rho} u \otimes u)-\mu \Delta u-(\lambda+\mu) \nabla \operatorname{div} u+\nabla P=0.
\end{cases}
\end{align}
In this paper, we are concerned with the Cauchy problem of the system \eqref{1} and \eqref{cns} in $\R^+\times \R^d$ together with the initial data
$(\tilde{\rho}, u, \theta)(t=0)=(\tilde{\rho}_0, u_0, \theta_0)$ and
$(\tilde{\rho}, u)(t=0)=(\tilde{\rho}_0, u_0)$, respectively.
We assume that the density of the fluid satisfies the non-vacuum condition
$$
\inf _{t \geq 0,\; x \in \R^d} \tilde{\rho}(t, x)>0 .
$$
We can check that if $(\tilde{\rho}, u, \theta)$ solves \eqref{1}, so does $\left(\tilde{\rho}_\ell(t, x), u_\ell(t, x), \theta_\ell(t, x)\right)$, where$$
\left(\tilde{\rho}_\ell(t, x), u_\ell(t, x), \theta_\ell(t, x)\right)=\left(\tilde{\rho}(\ell^2 t, \ell x), \ell u(\ell^2 t, \ell x), \ell^2 \theta(\ell^2 t, \ell x)\right), \quad \ell>0.
$$
This suggests us to choose initial data $(\tilde{\rho}_0, u_0, \theta_0)$ in ``critical spaces" whose norm is invariant for all $\ell>0$ (up to a constant independent of $\ell$) by the transformation $(\tilde{\rho}_0, u_0, \theta_0)(x) \mapsto(\tilde{\rho}_0(\ell x), \ell u_0(\ell x), \ell^2 \theta_0(\ell x))$.
It is natural that $\dot{H}^{{d}/{p}}_p\times\dot{H}^{{d}/{p}-1}_p\times\dot{H}^{{d}/{p}-2}_p$ and $\dot{B}_{p, q}^{{d}/{p}}\times\dot{B}_{p, q}^{{d}/{p}-1}\times\dot{B}_{p, q}^{{d}/{p}-2}$are critical spaces to  \eqref{1}.
Motivated by Fujita-Kato's theory on the incompressible Navier--Stokes equations \cite{fu} and based on the Littlewood-Paley theory, Danchin \cite{d1,d3,d4,d5,d6} studied the well-posedness for the compressible Navier-Stokes equations in critical Besov spaces (see also Chen-Miao-Zhang \cite{C1,C2} and Haspot \cite{H1,H2}). Roughly speaking, the system \eqref{1} is locally well-posed for the initial data
$$
\left(\tilde{\rho}_0-\bar{\rho}, u_0, \theta_0\right) \in \dot{B}_{p, 1}^{3 / p} \times \dot{B}_{p, 1}^{3 / p-1} \times \dot{B}_{p, 1}^{3 / p-2} \quad \text { with } p<3
$$
and the system \eqref{cns} is locally well-posed for the initial data
$$
\left(\tilde{\rho}_0-\bar{\rho}, u_0\right) \in \dot{B}_{p, 1}^{3 / p} \times\dot{B}_{p, 1}^{3 / p-1} \quad \text { with } p<6 .
$$
A natural question is whether or not the system \eqref{1} and \eqref{cns} are well-posed in the critical Besov spaces with $p \geq 3$ and $p\geq6$ respectively. Motivated by
Bourgain-Pavlovi\'{c} \cite{b08} and Germain \cite{g08} who proved the ill-posedness of the incompressible Navier--Stokes equations in the largest critical space $\dot{B}_{\infty, \infty}^{-1}$, Chen-Miao-Zhang \cite{C3} proved that the system \eqref{cns} is ill-posed in $\dot{B}_{p, 1}^{3/ p}\times \dot{B}_{p, 1}^{3/ p-1}$ with $p>6$.
These results as mentioned above suggest that the space $\dot{B}_{p, 1}^{d/ p}(\mathbb{R}^d) \times \dot{B}_{p, 1}^{d / p-1}(\mathbb{R}^d)$ with $p=2 d$ is the threshold space to \eqref{cns}. Indeed, in the case when $d=3$,  Chen-Wan \cite{CW} proved that the system \eqref{cns} is ill-posed in $\dot{B}_{p, 1}^{3/ p}\times \dot{B}_{6, 1}^{-1/2}$ with $p>6$.
Iwabuchi-Ogawa \cite{IO22} considered all the critical setting $p=2 d$ with $d \geq 2$ in both the density and velocity functions spaces $\dot{B}_{2d, 1}^{-1/2} \times \dot{B}_{2d, 1}^{-1/2}$ and proved that the system \eqref{cns} is ill-posed.

For the system \eqref{1}, Chen-Miao-Zhang \cite{C3} proved that \eqref{1} is ill-posed in the case when $d=3$ for the initial data
$$\left(\tilde{\rho}_0-\bar{\rho}, u_0, \theta_0\right) \in \dot{B}_{p, 1}^{3/ p}\times \dot{B}_{p, 1}^{3/ p-1}\times \dot{B}_{p, 1}^{3/ p-2}\quad \text { with } p>3.$$ Recently, Iwabuchi-Ogawa \cite{IO21} proved the ill-posedness of \eqref{1} in the case when $d=2$ for the initial data
$$
\left(\tilde{\rho}_0-\bar{\rho}, u_0, \theta_0\right) \in \dot{B}_{p, q}^{2 / p} \times \dot{B}_{p, q}^{2 / p-1} \times \dot{B}_{p, q}^{2 / p-2} \quad \text { with } (p,q)\in[1,\infty]\times [1,\infty).
$$
Aoki-Iwabuchi \cite{AI23} proved the ill-posedness of \eqref{1} in the case when $p=d=3$ for the initial data
$$
\left(\tilde{\rho}_0-\bar{\rho}, u_0, \theta_0\right) \in \dot{B}_{3, 1}^{1} \times \dot{B}_{3, 1}^{0} \times \dot{B}_{3, 1}^{-1}.
$$
Obviously, we remark that there remains a gap. Precisely speaking, for the end-point case $q=\infty$, i.e., in the weakest Besov space
$\dot{B}_{p, \infty}^{2 / p} \times \dot{B}_{p, \infty}^{2 / p-1} \times \dot{B}_{p, \infty}^{2 / p-2}$ with $1\leq p\leq\infty$, whether the system \eqref{1} is well-posed, as mentioned in \cite{IO21} (see p.575), is still an open problem. The goal of this paper is to answer the above question and we shall give a negative result that the Cauchy problem \eqref{1} is ill-posed.

In this paper, we write the density $\tilde{\rho}= 1 + \rho$ and rewrite \eqref{1} as follows.
\begin{align}\label{0}
\begin{cases}
\partial_t \rho+\operatorname{div} u+\operatorname{div}(\rho u)=0,   \\
\partial_t u-\mathcal{L} u+(1+\rho)(u \cdot \nabla) u+\nabla P=-\rho \partial_t u,  \\
\partial_t \theta+(1+\rho)(u \cdot \nabla) \theta+P \operatorname{div} u-\kappa \Delta \theta =2 \mu|\mathrm{D}(u)|^2+\lambda|\operatorname{div} u|^2-\rho \partial_t \theta, \\
\f(\rho,u,\theta\g)(t=0,x)=\f(\rho_0, u_0, \theta_0\g)(x),
\end{cases}
\end{align}
where we set
$$\mathcal{L} u=\mu \Delta u+(\mu+\lambda)\nabla\div\,u\quad\text{and}\quad P=(1 +\rho)\theta.$$

Now, we state our main result as follows.
\begin{theorem}\label{th1} Let $d\geq2$ and $1 \leq p \leq \infty$. The Cauchy problem \eqref{0} is ill-posed in $\dot{B}_{p, \infty}^{d/{p}}(\mathbb{R}^d) \times \dot{B}_{p, \infty}^{d/p-1}(\mathbb{R}^d) \times \dot{B}_{p, \infty}^{d/p-2}(\mathbb{R}^d)$. More precisely, there exist a sequence of initial data $\left\{u_{0, N}\right\}_N \subset \mathcal{S}(\mathbb{R}^d)$, positive time $\left\{T_N\right\}_N$ with $T_N \rightarrow 0$ as $N \rightarrow \infty$, and a sequence of corresponding smooth solutions $\left(\rho_N, u_N, \theta_N\right)$ in the time interval $\left[0, T_N\right]$ with the data $\left(0, u_{0, N}, 0\right)$ such that
$$
\|u_{0, N}\|_{\dot{B}_{p, \infty}^{d/{p}-1}}\thickapprox1
$$
and
\bbal
\lim _{N \rightarrow \infty}\left\|u\left(T_N\right)-u_{0, N}\right\|_{\dot{B}_{p, \infty}^{d/p-1}}\geq \eta_0
\end{align*}
for some positive constant $\eta_0$.
\end{theorem}
\begin{remark}\label{re1}
Theorem \ref{th1} demonstrates the discontinuous of the solution for the Navier--Stokes equations \eqref{0} in $\dot{B}_{p, \infty}^{d/{p}}(\mathbb{R}^d) \times \dot{B}_{p, \infty}^{d/p-1}(\mathbb{R}^d) \times \dot{B}_{p, \infty}^{d/p-2}(\mathbb{R}^d)$. More precisely, there exists a sequence of initial data $u_{0,N}\in \dot{B}_{p, \infty}^{d/p-1}(\mathbb{R}^d)$ such that the sequence of corresponding smooth solutions $u\left(T_N\right)$ of \eqref{0} that start from $u_{0,N}$ do not converge back to $u_{0,N}$ in the metric of $\dot{B}_{p, \infty}^{d/p-1}(\mathbb{R}^d)$ as time goes to zero.
\end{remark}
To the beset of our knowledge, it is also an open question whether or not the Cauchy problem for incompressible Navier--Stokes equations is well-posed  in $\dot{B}_{p, \infty}^{d/p-1}(\mathbb{R}^d)$ with $1 \leq p < \infty$. As a by-product, we have
\begin{corollary}\label{cor}Let $d\geq2$ and $1 \leq p \leq \infty$. The Cauchy problem for the incompressible Navier--Stokes equations is ill-posed in $\dot{B}_{p, \infty}^{d/p-1}(\mathbb{R}^d)$ .
\end{corollary}
\noindent\textbf{Organization of our paper.} In Section \ref{sec2}, we list some notations and known results which will be used in the sequel. In Section \ref{sec3}, we just prove Theorem \ref{th1} since the proof of Corollary \ref{cor} can be done with minor modifications.
\section{Preliminary}\label{sec2}
We will use the following notations throughout this paper.
For $X$ a Banach space and $I\subset\R$, we denote by $\mathcal{C}(I;X)$ the set of continuous functions on $I$ with values in $X$. Sometimes we will denote $L^p(0,T;X)$ by $L_T^pX$.
Let us recall that for all $f\in \mathcal{S}'$, the Fourier transform $\mathcal{F}f$, also denoted by $\widehat{f}$, is defined by
$
\mathcal{F}f(\xi)=\widehat{f}(\xi)=\int_{\R^d}e^{-\mathrm{i}x\cd \xi}f(x)\dd x$ for any $\xi\in\R^d.$
The inverse Fourier transform allows us to recover $f$ from $\widehat{f}$:
$
f(x)=\mathcal{F}^{-1}\widehat{f}(x)=(2\pi)^{-d}\int_{\R^d}e^{\mathrm{i}x\cdot\xi}\widehat{f}(\xi)\dd\xi.
$
Next, we will recall some facts about the Littlewood-Paley decomposition and the nonhomogeneous Besov spaces (see \cite{BCD} for more details).
Choose a radial, non-negative, smooth function $\vartheta:\R^d\mapsto [0,1]$ such that
${\rm{supp}} \;\vartheta\subset B(0, 4/3)$ and
$\vartheta(\xi)\equiv1$ for $|\xi|\leq3/4$.
Setting $\varphi(\xi):=\vartheta(\xi/2)-\vartheta(\xi)$, then we deduce that $\varphi$ has the following properties
\begin{itemize}
  \item ${\rm{supp}} \;\varphi\subset \left\{\xi\in \R^d: 3/4\leq|\xi|\leq8/3\right\}$;
  \item $\varphi(\xi)\equiv 1$ for $4/3\leq |\xi|\leq 3/2$;
  \item $\vartheta(\xi)+\sum_{j\geq0}\varphi(2^{-j}\xi)=1$ for any $\xi\in \R^d$;
  \item $\sum_{j\in \mathbb{Z}}\varphi(2^{-j}\xi)=1$ for any $\xi\in \R^d\setminus\{0\}$.
\end{itemize}
The  homogeneous dyadic blocks are defined as follows
\begin{align*}
\forall\, f\in \mathcal{S}'_h(\R^d),\quad
\dot{\Delta}_jf=\varphi(2^{-j}D)f,\; \; \text{if}\;j\in \mathbb{Z},
\end{align*}
where the pseudo-differential operator is defined by $\sigma(D):f\to\mathcal{F}^{-1}(\sigma \mathcal{F}f)$ and $\mathcal{S}'_h$ is given by
\begin{eqnarray*}
\mathcal{S}'_h:=\f\{f \in \mathcal{S'}(\mathbb{R}^{d}):\; \lim_{j\rightarrow-\infty}\|\vartheta(2^{-j}D)f\|_{L^{\infty}}=0 \g\}.
\end{eqnarray*}

We recall the definition of the Besov Spaces and norms.
\begin{definition}[\cite{BCD}]
Let $s\in\mathbb{R}$ and $(p,q)\in[1, \infty]^2$. The  homogeneous Besov spaces is defined as follows
$$
\dot{B}^{s}_{p,q}=\f\{f\in \mathcal{S}'_h:\;\|f\|_{\dot{B}^{s}_{p,q}(\R^d)}<\infty\g\},
$$
where
\begin{align*}{\|f\|_{\dot{B}^{s}_{p,q}(\R^d)}=}\begin{cases}
\left\{\sum\limits_{j\in\mathbb{Z}}\f(2^{sjq}\|\dot{\Delta}_jf\|^q_{L^p(\R^d)}\g)\right\}^{1/q}, &\text{if}\; 1\leq q<\infty,\\
\sup\limits_{j\in\mathbb{Z}}2^{sj}\|\dot{\Delta}_jf\|_{L^p(\R^d)}, &\text{if}\; q=\infty.
\end{cases}
\end{align*}
\end{definition}
The following Bernstein's inequalities will be used in the sequel.
\begin{lemma}[\cite{BCD}] \label{lem2.1} Let $\mathcal{B}$ be a ball and $\mathcal{C}$ be an annulus. There exists a constant $C>0$ such that for all $k\in \mathbb{Z}^+\cup \{0\}$, any $\lambda\in \R^+$ and any function $f\in L^p$ with $1\leq p \leq q \leq \infty$, we have
\begin{align*}
&{\rm{supp}}\widehat{f}\subset \lambda \mathcal{B}\;\Rightarrow\; \|D^kf\|_{L^q}\leq C^{k+1}\lambda^{k+(\frac{1}{p}-\frac{1}{q})}\|f\|_{L^p},  \\
&{\rm{supp}}\widehat{f}\subset \lambda \mathcal{C}\;\Rightarrow\; C^{-k-1}\lambda^k\|f\|_{L^p} \leq \|D^kf\|_{L^p} \leq C^{k+1}\lambda^k\|f\|_{L^p}.
\end{align*}
\end{lemma}
Next we recall the following product law which will be used often in the sequel.
\begin{lemma}[\cite{BCD}]\label{bana} $\dot{B}^{{d}/{p}}_{p,1}(\R^d)$ is a Banach algebra. Furthermore, there exists a positive constant $c_M$ such that
$$\|fg\|_{\dot{B}^{{d}/{p}}_{p,1}(\R^d)}\leq c_M\|f\|_{\dot{B}^{{d}/{p}}_{p,1}(\R^d)}\|g\|_{\dot{B}^{{d}/{p}}_{p,1}(\R^d)},\quad \forall (f,g)\in \dot{B}^{{d}/{p}}_{p,1}(\R^d).$$
\end{lemma}

Finally, we recall the regularity estimates for the heat equations.
\begin{lemma}[\cite{BCD}]\label{rg}
Let $s\in \mathbb{R}$, $1\leq p,r\leq \infty$  and $1\leq q_1\leq q_2\leq \infty$. Assume that $u_0\in \dot{B}^s_{p,r}$ and $f\in {\LL}^{q_1}_T(\dot{B}^{s+{2}/{q_1}-2}_{p,r})$. Then the heat equations
\begin{align*}
\begin{cases}
\partial_tu-\kappa\Delta u=f,\\
 u(t=0)=u_0,
\end{cases}
\end{align*}
has a unique solution $u\in \LL^{q_2}_T(\dot{B}^{s+{2}/{q_2}}_{p,r}))$ satisfying for all $T>0$
\begin{align*}
\|u\|_{\LL^{q_2}_T\dot{B}^{s+{2}/{q_2}}_{p,r}}\leq C\left(\|u_0\|_{\dot{B}^s_{p,r}}+\|f\|_{{\LL}^{q_1}_T\dot{B}^{s+{2}/{q_1}-2}_{p,r}}\right).
\end{align*}
In particular, there holds
\begin{align*}
\|u\|_{\LL^{\infty}_T\dot{B}^{s}_{p,r}}+\|\kappa \Delta u\|_{\LL^{1}_T\dot{B}^{s}_{p,r}}+\|\pa_tu\|_{\LL^1_T\dot{B}^{s}_{p,r}}\leq C\left(\|u_0\|_{\dot{B}^s_{p,r}}+\|f\|_{\LL^1_T\dot{B}^{s}_{p,r}}\right).
\end{align*}
\end{lemma}

\section{Proof of Theorem \ref{th1}}\label{sec3}

In this section we prove our main result.
\subsection{Example for Initial Data}\label{subsec3.1}
The choice of initial data for showing the discontinuity is the most influencing. Before constructing the sequence of initial data, we need to introduce smooth, radial cut-off functions to localize the frequency region. Let $\widehat{\phi}\in \mathcal{C}^\infty_0(\mathbb{R})$ be an even, real-valued and non-negative function on $\R$ and satisfy
\begin{align*}
{\widehat{\phi}(\xi)=}\begin{cases}
1, &\text{if}\; |\xi|\leq \frac{1}{200d},\\
0, &\text{if}\; |\xi|\geq \frac{1}{100d}.
\end{cases}
\end{align*}
For any $p\in[1,\infty]$, there exists two positive constants $C_1$ and $C_2$ such that
\begin{align*}
C_1\leq\|\phi\|_{L^p(\R)}\leq C_2.
\end{align*}
{\bf Definition}\, Let $d\geq2$ and $p\in[1,\infty]$. For some fixed $\delta\in(0,1)$ which will be determined later and for any $N\in \mathbb{Z}^+$, we define $\left(\rho_0, u_0, \theta_0\right)=\left(0, u_{0,N}, 0\right)$  by
\begin{align*}
&u_{0,N}(x):=\delta2^{-\fr{d}{p}N}
\f(-\pa_2f_N(x),\;
\pa_1f_N(x),\;
\underbrace{0,\;
\cdots,\;
0}_{d-2}
\g),\quad \text{where}\\
&f_N(x):=\phi(x_1)\cos \left(\frac{17}{12}2^Nx_1\right)\prod_{i=2}^d\phi(x_i).
\end{align*}

\begin{lemma}\label{le1} Let $d\geq2$ and $(p,q)\in[1,\infty]^2$. The divergence-free vector field $u_{0,N}$ is defined as above.
Then for $\sigma\in \R$, there exists some sufficiently large $N\in \mathbb{Z}^+$ and some sufficiently enough $\delta>0$  such that
\begin{align}
&\|u_{0,N}\|_{\dot{B}^{d/p-1+\sigma}_{p,q}}\thickapprox \delta2^{\sigma N}\|\phi\|^d_{L^{p}}, \label{z1}\\
&2^{N(d/p-1)}\|\De_{N}u_{0,N}\|_{L^p}\thickapprox \delta\|\phi\|^d_{L^{p}}. \label{z2}
\end{align}
\end{lemma}
\begin{proof}\,Straightforward computations yield
\bbal
\widehat{f_N}(\xi)=
\left[\widehat{\phi}\left(\xi_1-\frac{17}{12}2^N\right)+\widehat{\phi}\left(\xi_1+\frac{17}{12}2^N\right)\right]\prod_{i=2}^d\widehat{\phi}(\xi_i),
\end{align*}
which implies
\bal\label{s}
\mathrm{supp} \ \widehat{u_{0,N}}\subset \left\{\xi\in\R^d: \ \frac{4}{3}2^N\leq |\xi|\leq \frac{3}{2}2^N\right\}.
\end{align}
Due to the fact $\varphi(\xi)\equiv 1$ for $\frac43\leq |\xi|\leq \frac32$, namely,
\bbal
\varphi\f(2^{-N}\xi\g)\equiv 1\quad \text{in}\;  \f\{\xi\in\R^d: \ \frac{4}{3}2^{N}\leq |\xi|\leq \frac{3}{2}2^{N}\g\},
\end{align*}
we have
\begin{align}\label{a}
{\Delta_j(f_N)=\mathcal{F}^{-1}\f(\varphi(2^{-j}\cdot)\widehat{f}_{N}\g)=}\begin{cases}
f_N, &\text{if}\; j=N,\\
0, &\text{if}\; j\neq N.
\end{cases}
\end{align}
In particular, it holds that for some sufficiently large $N\in \mathbb{Z}^+$
 \begin{align*}
 \dot{\Delta}_{N}u_{0,N}=\delta2^{-\fr{d}{p}N}
\f(
-\pa_2f_N,\;
\pa_1f_N,\;
0,\;
\cdots,\;
0
\g).
    \end{align*}
Then we deduce that
\begin{align*}
 2^{N(d/p-1)}\|\dot{\Delta}_{N}u_{0,N}\|_{L^p(\R^d)}&
= \delta2^{-N}\f(\|\pa_1f_N\|_{L^{p}(\R^d)}+\|\pa_2f_N\|_{L^{p}(\R^d)}\g)\thickapprox \delta\|\phi\|^d_{L^{p}},
\end{align*}
from which and \eqref{a}, it follows that
\begin{align*}
 \|u_{0,N}\|_{\dot{B}_{p,q}^{d/p-1+\sigma}(\R^d)}&=2^{N(d/p-1+\sigma)}\|\dot{\Delta}_Nu_{0,N}\|_{L^p(\R^d)}\thickapprox \delta2^{\sigma N}\|\phi\|^d_{L^{p}}.
\end{align*}
This completes the proof of Lemma \ref{le1}.
\end{proof}
\subsection{Reformulation of the equation}\label{subsec3.2}
Motivated by \cite{IO21,AI23}, we introduce the following integral equations of \eqref{0}.
$$
\left\{\begin{aligned}
&\rho(t)  =-\int_0^t \left\{\operatorname{div}u(\tau)+\operatorname{div}\left(\rho u(\tau)\right)\right\} \dd \tau, \\
&u(t) =e^{t \mathcal{L}} u_0-\int_0^t e^{(t-\tau) \mathcal{L}}\left\{(1+\rho)(u \cdot \nabla) u(\tau)+\nabla P(\tau)+\rho \partial_\tau u(\tau)\right\} \dd \tau, \\
&\theta(t)= -\int_0^t e^{(t-\tau)\kappa \Delta} \left\{(1+\rho)(u \cdot \nabla) \theta(\tau)+P\operatorname{div} u(\tau)+\rho\partial_\tau \theta(\tau)
-2 \mu|\mathrm{D}(u(\tau))|^2-\lambda|\operatorname{div} u(\tau)|^2\right\} \dd \tau,
\end{aligned}\right.
$$
where $P=(1+\rho) \theta$. Based on the expansion of the solution for initial data $\left(0, \varepsilon u_{0,N}, 0\right)$ with small parameter $\varepsilon$
$$
(\rho, u, \theta)=\left(P_0, U_0, \Theta_0\right)+\varepsilon\left(P_1, U_1, \Theta_1\right)+\varepsilon^2\left(P_2, U_2, \Theta_2\right)+\cdots,
$$
we define the expansion $\left\{\left(P_k, U_k, \Theta_k\right)\right\}_{k=0}^{\infty}$ for the initial data $(0,u_{0,N},0)$
\begin{align}\label{y1}
& \begin{cases}
P_0(t) \equiv 0, \\
P_1(t) \equiv-\int_0^t \operatorname{div} U_1(\tau) \dd \tau, \\
P_2(t) \equiv-\int_0^t \left\{\operatorname{div} U_2(\tau)+\operatorname{div}\left(P_{1}(\tau) U_{1}(\tau)\right)\right\} \dd \tau, \\
P_k(t) \equiv-\int_0^t\left\{\operatorname{div} U_k(\tau)+\sum\limits_{k_1+k_2=k} \operatorname{div}\left(P_{k_1}(\tau) U_{k_2}(\tau)\right)\right\} \dd \tau, \quad k \geq 3,
\end{cases}
\end{align}
\begin{align}\label{y2}
& \begin{cases}
U_0(t) \equiv 0, \\
U_1(t) \equiv e^{t \mathcal{L}} u_{0,N}, \\
U_2(t) \equiv -\int_0^t e^{(t-\tau) \mathcal{L}}\left\{U_1(\tau) \cdot \nabla U_1(\tau)+\nabla \Theta_2(\tau)+P_1(\tau) \partial_\tau U_1(\tau)\right\} \dd \tau, \\
U_k(t) \equiv -\int_0^t e^{(t-\tau) \mathcal{L}}\left\{\sum\limits_{k_1+k_2=k}U_{k_1}(\tau) \cdot \nabla U_{k_2}(\tau)+\sum\limits_{k_1+k_2+k_3=k} P_{k_3}(\tau)U_{k_1}(\tau) \cdot \nabla U_{k_2}(\tau)\right.\\
\quad\quad\quad+\nabla \f(\Theta_k(\tau)+\sum\limits_{k_1+k_2=k} P_{k_1}(\tau) \Theta_{k_2}(\tau)\g)
\left.+\sum\limits_{k_1+k_2=k} P_{k_1}(\tau) \partial_\tau U_{k_2}(\tau)\right\} \dd \tau, \quad k \geq 3,
\end{cases}
\end{align}
\begin{align}\label{y3}
& \begin{cases}
\Theta_0(t) \equiv \Theta_1(t) \equiv 0, \\
\Theta_2(t) \equiv \int_0^t e^{(t-\tau) \kappa \Delta}\left\{\widetilde{D}\left(U_1(\tau)\right): \widetilde{D}\left(U_1(\tau)\right)\right\}\dd \tau, \\
\Theta_k(t) \equiv -\int_0^t e^{(t-\tau) \kappa \Delta}\left\{\sum\limits_{k_1+k_2=k}U_{k_1}(\tau) \cdot \nabla \Theta_{k_2}(\tau)\right. +\sum\limits_{k_1+k_2+k_3=k} P_{k_3}(\tau)U_{k_1}(\tau) \cdot \nabla \Theta_{k_2}(\tau) \\
\quad\quad\quad+\sum\limits_{k_1+k_2=k} \Theta_{k_1}(\tau) \operatorname{div} U_{k_2}(\tau)+\sum\limits_{k_1+k_2+k_3=k} P_{k_3}(\tau) \Theta_{k_1}(\tau) \operatorname{div} U_{k_2}(\tau) \\
\quad\quad\quad+\sum\limits_{k_1+k_2=k} P_{k_1}(\tau) \partial_\tau \Theta_{k_2}(\tau)
\left.-\sum\limits_{k_1+k_2=k} \widetilde{D}\left(U_{k_1}(\tau)\right): \widetilde{D}\left(U_{k_2}(\tau)\right)\right\} \dd \tau, \quad k \geq 3,\end{cases}
\end{align}
where we have used the notations $\widetilde{D}(u): \widetilde{D}(v) \equiv 2 \mu \operatorname{tr}\big(\mathrm{D} u \cdot \mathrm{D}^\top v\big)+\lambda(\operatorname{div} u \cdot \operatorname{div} v)$ and
$$
\sum_{k_1+k_2=k} a_{k_1} b_{k_2}=\sum_{k_1=1}^{k-1} a_{k_1} b_{k-k_1} \quad\text{and}\quad \sum_{k_1+k_2+k_3=k} a_{k_1} b_{k_2} c_{k_3}=\sum_{k_1=1}^{k-1} \sum_{k_2=1}^{k-k_1} a_{k_1} b_{k_2} c_{k-\left(k_1+k_2\right)}.
$$
Then we have the expansion of solution as follows:
\bal\label{zk}
\rho(t)=\left.\sum_{k=1}^{\infty} \varepsilon^k P_k(t)\right|_{\varepsilon=1}, \quad u(t)=\left.\sum_{k=1}^{\infty} \varepsilon^k U_k(t)\right|_{\varepsilon=1}, \quad \theta(t)=\left.\sum_{k=1}^{\infty} \varepsilon^k \Theta_k(t)\right|_{\varepsilon=1}
\end{align}
for the initial data $(\rho, u, \theta)(t=0)=\left(0, u_{0,N}, 0\right)$.
\subsection{Estimation for ($P_k,U_k,\Theta_k$)}\label{subsec3.3}
\begin{proposition}\label{pro1} Let $u_{0,N}$ be defined as above. Then there exist $C_0>0$ such that for all $t \leq 2^{-2 N}$,
\begin{align}
 \left(t 2^N\right)^{-1}\left\|P_k(t)\right\|_{\dot{B}^{d/p}_{p,1}}+\left\|U_k(t)\right\|_{\dot{B}^{d/p}_{p,1}} +2^{-N}\left\|\Theta_k(t)\right\|_{\dot{B}^{d/p}_{p,1}} \leq C_0^{k-1} t^{k-1} 2^{(2k-1) N} \delta^k .
\end{align}
\end{proposition}
\begin{proof}
Due to support condition \eqref{s} of $u_{0,N}$, one has
\begin{align*}
\operatorname{supp} \widehat{P_k},\, \operatorname{supp} \widehat{U_k},\, \operatorname{supp} \widehat{\Theta_k}\
\subset\left\{\xi \in \mathbb{R}^d: \,|\xi| \leq k 2^{N+1}\right\}, \quad k=1,2,3, \cdots .
\end{align*}

{\bf In the case when $k=1$}. Due to $\eqref{y2}_2$, one has
\begin{align}\label{u1}
& \begin{cases}
\pa_tU_1-\mu \Delta U_1-(\mu+\lambda)\nabla\div\,U_1=0,\\
U_1(t=0)=u_{0,N},\\
\div\,u_{0,N}=0.
\end{cases}\end{align}
Obviously, it holds that $\div\,U_1=0$. Thus \eqref{u1} reduces to
\begin{align*}
& \begin{cases}
\pa_tU_1-\mu \Delta U_1=0,\\
U_1(t=0)=u_{0,N},\\
\div\,u_{0,N}=0.
\end{cases}\end{align*}
From Duhamel's principle, it follows that
\begin{align}\label{u2}
& \begin{cases}
P_1(t)=0, \\
U_1(t)=e^{\mu t \Delta} u_{0,N}, \\
\Theta_1(t)= 0.
\end{cases}\end{align}
By Lemma \ref{rg} and \eqref{u2}, we deduce that
\begin{align}\label{u3}
& \left\|U_1(t)\right\|_{\dot{B}^{d/p}_{p,1}} \leq C\|u_{0,N}\|_{\dot{B}^{d/p}_{p,1}} \leq C 2^N\delta, \\
& \left\|P_1(t)\right\|_{\dot{B}^{d/p}_{p,1}}+\left\|\Theta_1(t)\right\|_{\dot{B}^{d/p}_{p,1}} =0.
\end{align}

{\bf In the case when $k=2$}. Due to \eqref{y1}-\eqref{y3}, one has
\begin{align}
& \begin{cases}
P_2(t)=-\int_0^t \operatorname{div} U_2(\tau) \dd \tau, \\
U_2(t)=-\int_0^t e^{(t-\tau) \mathcal{L}}\left(U_1(\tau) \cdot \nabla U_1(\tau)+\nabla \Theta_2(\tau)\right) \dd \tau, \\
\Theta_2(t)= \int_0^t e^{(t-\tau) \kappa \Delta}\left(2 \mu \operatorname{tr}\left(\mathrm{D} U_1(\tau) \cdot \mathrm{D}^\top U_1(\tau)\right)\right) \dd \tau, \\
\end{cases}
\end{align}
it follows from Lemma \ref{rg} and \eqref{u3} that
\begin{align*}
\left\|\Theta_2(t)\right\|_{\dot{B}^{d/p}_{p,1}} &\leq  C \int_0^t \left\|\nabla U_1(\tau)\right\|_{\dot{B}^{d/p}_{p,1}}^2 \dd \tau \leq  C \int_0^t 2^{2 N}\left\|U_1(\tau)\right\|_{\dot{B}^{d/p}_{p,1}}^2 \dd \tau \leq C t 2^{4 N} \delta^2, \\
\left\|U_2(t)\right\|_{\dot{B}^{d/p}_{p,1}}&\leq C \int_0^t\left(\left\|U_1(\tau)\right\|_{\dot{B}^{d/p}_{p,1}}\left\|\nabla U_1(\tau)\right\|_{\dot{B}^{d/p}_{p,1}}+2^N\left\|\Theta_2(\tau)\right\|_{\dot{B}^{d/p}_{p,1}}
\right) \dd \tau \\
&\leq  C\left(t 2^{3 N} \delta^2+t^2 2^{5 N} \delta^2\right) \leq C t 2^{3 N} \delta^2, \quad t \leq T= 2^{-2 N}, \\
\left\|P_2\right\|_{\dot{B}^{d/p}_{p,1}}  &\leq  C \int_0^t2^N\left\|U_2(\tau)\right\|_{\dot{B}^{d/p}_{p,1}}
 \dd \tau \leq C t^2 2^{4 N} \delta^2.
\end{align*}

{\bf In the case when $k\geq3$}. Firstly, we need to estimate $\partial_t U_k$ and $\partial_t \Theta_k$. By Lemma \ref{rg}, it is straightforward to verify that for $t\leq T$ with $T=2^{-2 N}$
$$
\begin{gathered}
\left\|\partial_t \Theta_2\right\|_{L^1_t \dot{B}^{d/p}_{p,1}}+\left\|\kappa\Delta\Theta_2\right\|_{L^1_t \dot{B}^{d/p}_{p,1}} \leq C t 2^{4 N} \delta^2, \\
\left\|\partial_t U_2\right\|_{L^1_t \dot{B}^{d/p}_{p,1}}+\left\|\mathcal{L} U_2\right\|_{L^1_t \dot{B}^{d/p}_{p,1}} \leq C t 2^{3 N} \delta^2.
\end{gathered}
$$

It remains to prove the following Lemma to complete the proof of Proposition \ref{pro1}.
\end{proof}
\begin{lemma}\label{le2}There exist positive constants $c_1$ and $c_2$ such that for all $t<T= 2^{-2 N}$,
\begin{align*}
&\left\|P_k(t)\right\|_{\dot{B}^{d/p}_{p,1}} \leq c_1^{k-1} c_2^k t^k 2^{2k N} \delta^k(1+k)^{-4}, \\
&\left\|U_k(t)\right\|_{\dot{B}^{d/p}_{p,1}}+\left\|\partial_t U_k(\cdot)\right\|_{L^1_t\dot{B}^{d/p}_{p,1}}+\left\|\mathcal{L} U_k(\cdot)\right\|_{L^1_t\dot{B}^{d/p}_{p,1}}
\leq c_1^{k-1} c_2^{k-1} t^{k-1} 2^{(2k-1) N} \delta^k(1+k)^{-4}, \\
&\left\|\Theta_k(t)\right\|_{\dot{B}^{d/p}_{p,1}}+\left\|\partial_t \Theta_k(\cdot)\right\|_{L^1_t\dot{B}^{d/p}_{p,1}}+\left\|\kappa \Delta \Theta_k(\cdot)\right\|_{L^1_t\dot{B}^{d/p}_{p,1}}
\leq c_1^{k-1} c_2^k t^{k-1} 2^{2k N} \delta^k(1+k)^{-4} .
\end{align*}
\end{lemma}
\begin{proof} We assume the estimate for $1,2, \ldots k-1$ with $k \geq 3$, and show the case of $k$ to apply an induction argument. We notice from the definition of $(P_k, U_k, \Theta_k)$ that $\Theta_k$ can be handled by the assumption for $1,2, \ldots, k-1$, while $U_k$ needs the estimate of $\Theta_k$ and $P_k$ needs the estimate of $U_k$.

We emphasize that $c_M$ appeared in Lemma \ref{bana}. From now on, we denote by $c_2$ the biggest absolute constant such that the inequalities before Proposition \ref{pro1} holds. We will also apply
$$
\frac{1}{k-1} \sum_{k_1+k_2=k} k_1 k_2\left(1+k_1\right)^{-4}\left(1+k_2\right)^{-4} \leq \frac{C}{(1+k)^4},
$$
and take $c_1$ such that $c_1>\max\{10 C (c_M+1),c_M,\,1\}$.

Using Lemmas \ref{rg} and \ref{bana}, one has
$$
\begin{aligned}
&\frac{1}{c_M}\f\{\left\|\Theta_k(t)\right\|_{\dot{B}^{d/p}_{p,1}} +\left\|\partial_t \Theta_k(t)\right\|_{L^1_t \dot{B}^{d/p}_{p,1}}+\left\|\kappa \Delta \Theta_k(t)\right\|_{L^1_t \dot{B}^{d/p}_{p,1}}\g\} \\
\leq &~ \int_0^t\left\{ c_2\sum_{k_1+k_2=k}\|U_{k_1}\|_{\dot{B}^{d/p}_{p,1}} k_2 2^N\|\Theta_{k_2}\|_{\dot{B}^{d/p}_{p,1}}\right.\\
&~ +c_M c_2 \sum_{k_1+k_2+k_3=k}\|P_{k_3}\|_{\dot{B}^{d/p}_{p,1}}\|U_{k_1}\|_{\dot{B}^{d/p}_{p,1}} k_2 2^N\|\Theta_{k_2}\|_{\dot{B}^{d/p}_{p,1}} \\
&~+ c_2 \sum_{k_1+k_2=k}\|\Theta_{k_1}\|_{\dot{B}^{d/p}_{p,1}} k_2 2^N\|U_{k_2}\|_{\dot{B}^{d/p}_{p,1}}\\
 &~+c_M c_2 \sum_{k_1+k_2+k_3=k}\|P_{k_3}\|_{\dot{B}^{d/p}_{p,1}}\|\Theta_{k_1}\|_{\dot{B}^{d/p}_{p,1}} k_2 2^N\|U_{k_2}\|_{\dot{B}^{d/p}_{p,1}} \\
&~+ \sum_{k_1+k_2=k}\|P_{k_1}\|_{\dot{B}^{d/p}_{p,1}}\|\partial_\tau \Theta_{k_2}\|_{\dot{B}^{d/p}_{p,1}}\\
&~\left.+ c_2^2 \sum_{k_1+k_2=k} k_1 2^N\|U_{k_1}\|_{\dot{B}^{d/p}_{p,1}} k_2 2^N\|U_{k_2}\|_{\dot{B}^{d/p}_{p,1}}\right\} \dd \tau.
\end{aligned}
$$
We estimate the second term as follows.
\begin{align*}
& \int_0^t  c_M c_2 \sum_{k_1+k_2+k_3=k}\|P_{k_3}\|_{\dot{B}^{d/p}_{p,1}}\|U_{k_1}\|_{\dot{B}^{d/p}_{p,1}} k_2 2^N\|\Theta_{k_2}\|_{\dot{B}^{d/p}_{p,1}} \dd \tau \\
\leq &~\int_0^t \tau^{k-2} \dd \tau \cdot  c_M c_2 c_1^{k-3} c_2^{k-1} 2^{2k N} \delta^k \sum_{k_1+k_2+k_3=k} k_2\left(1+k_1\right)^{-4}\left(1+k_2\right)^{-4}\left(1+k_3\right)^{-4}  \\
\leq &~\frac{1}{k-1} c_Mc_1^{k-3}  c_2^k t^{k-1} 2^{2k N} \delta^k \sum_{k_1+k_2=k} k_2\left(1+k_1\right)^{-4}\left(1+k_2\right)^{-4}\left(1+k_3\right)^{-4}\\
\leq &~\frac{Cc_M}{c_1} c_1^{k-2}  c_2^k t^{k-1} 2^{2k N} \delta^k(1+k)^{-4}\leq\frac{C}{c_1} c_1^{k-1} c_2^k t^{k-1} 2^{2k N}\delta^k(1+k)^{-4} .
\end{align*}
We estimate the final term as follows.
$$
\begin{aligned}
& \int_0^t  c_2^2 \sum_{k_1+k_2=k} k_1 2^N\|U_{k_1}\|_{\dot{B}^{d/p}_{p,1}} k_2 2^N\|U_{k_2}\|_{\dot{B}^{d/p}_{p,1}} \dd \tau \\
\leq &~\int_0^t \tau^{k-2} \dd \tau \cdot  c_2^2 c_1^{k-2} c_2^{k-2} 2^{2k N} \delta^k \sum_{k_1+k_2=k} k_1 k_2\left(1+k_1\right)^{-4}\left(1+k_2\right)^{-4} \\
\leq &~\frac{1}{k-1} c_1^{k-2}  c_2^k t^{k-1} 2^{2k N} \delta^k \sum_{k_1+k_2=k} k_1 k_2\left(1+k_1\right)^{-4}\left(1+k_2\right)^{-4}\\
\leq &~C c_1^{k-2}  c_2^k t^{k-1} 2^{2k N} \delta^k(1+k)^{-4}=\frac{C}{c_1} c_1^{k-1} c_2^k t^{k-1} 2^{2k N}\delta^k(1+k)^{-4} .
\end{aligned}
$$
By estimating other terms and the assumption of the induction, we have
$$
\left\|\Theta_k\right\|_{L^{\infty}_t \dot{B}^{d/p}_{p,1}} \leq \frac{10 C (c_M+1)}{c_1} c_1^{k-1} c_M c_2^k t^{k-1} 2^{2k N} \delta^k(1+k)^{-4} .
$$
Notice that $10 C (c_M+1) / c_1 \leq 1$, then we conclude the estimate for $\Theta_k$. We can also show the estimate for $U_k$ in an analogous way to $\Theta_k$.
In fact
$$
\begin{aligned}
&\|U_k(t)\|_{\dot{B}^{d/p}_{p,1}} +\|\partial_t U_k(t)\|_{L^1_t \dot{B}^{d/p}_{p,1}}+\left\|\kappa U_k(t)\right\|_{L^1_t \dot{B}^{d/p}_{p,1}} \\
\leq &~ \int_0^t\left\{c_M c_2\sum_{k_1+k_2=k}\|U_{k_1}\|_{\dot{B}^{d/p}_{p,1}} k_2 2^N\|U_{k_2}\|_{\dot{B}^{d/p}_{p,1}}\right.\\
 &~+c_M^2 c_2 \sum_{k_1+k_2+k_3=k}\|P_{k_3}\|_{\dot{B}^{d/p}_{p,1}}\|U_{k_1}\|_{\dot{B}^{d/p}_{p,1}} k_2 2^N\|U_{k_2}\|_{\dot{B}^{d/p}_{p,1}} \\
&~+kc_22^N\|\Theta_{k}\|_{\dot{B}^{d/p}_{p,1}}  +c_Mc_2k2^N \sum_{k_1+k_2=k}\|P_{k_1}\|_{\dot{B}^{d/p}_{p,1}}\|\Theta_{k_2}\|_{\dot{B}^{d/p}_{p,1}}\\
 &~\left.+c_M \sum_{k_1+k_2=k}\|P_{k_1}\|_{\dot{B}^{d/p}_{p,1}}\|\partial_\tau U_{k_2}\|_{\dot{B}^{d/p}_{p,1}}\right\} \dd \tau.
\end{aligned}
$$
The modification of the constant also appears in the estimate of $P_k$. We take $c_1>1$ and $c_2>1$ such that $c_1 c_2>5 C\left(c_1+c_M\right)$. By the assumption of the induction, we see that
$$
\begin{aligned}
\left\|P_k(t)\right\|_{\dot{B}^{d/p}_{p,1}} & \leq \int_0^t\|\operatorname{div} U_k(s)\|_{\dot{B}^{d/p}_{p,1}} \dd s+ \int_0^t\sum_{k_1+k_2=k} \|\operatorname{div}\left(P_{k_1}(s) U_{k_2}(s)\right)\|_{\dot{B}^{d/p}_{p,1}} \dd s \\
& \leq C k 2^N \int_0^t\|U_k(s)\|_{\dot{B}^{d/p}_{p,1}} \dd s
+C k 2^N \int_0^t\sum_{k_1+k_2=k} \|P_{k_1}(s) U_{k_2}(s)\|_{\dot{B}^{d/p}_{p,1}} \dd s  \\
& \leq C c_1^{k-1} c_2^{k-1} t^k 2^{2k N} \delta^k(1+k)^{-4}+Cc_M c_1^{k-2}  c_2^{k-1} t^k2^{2k N} \delta^k \sum_{k_1+k_2=k}\left(1+k_1\right)^{-4}\left(1+k_2\right)^{-4} \\
& \leq \frac{C\left(c_1+c_M\right)}{c_1 c_2} c_1^{k-1} c_2^k t^k 2^{k N}\delta^k(1+k)^{-4}  \leq c_1^{k-1} c_2^k t^k 2^{2k N} \delta^k(1+k)^{-4},
\end{aligned}
$$
where we have used $5 C\left(c_1+c_M\right)<c_1 c_2$. Then we completed the proof of Lemma \ref{le2}.
\end{proof}

{\bf Completion of the proof of Theorem \ref{th1}}\;
Using the triangle inequality, we obtain from \eqref{zk} that
\bal\label{lf1}
\|u(t)-u_{0,N}\|_{\dot{B}^{d/{p}-1}_{p,\infty}}
&\geq2^{{N(\fr{d}{p}-1)}}\big\|\De_{N}\big(U_1-u_{0,N}\big)\big\|_{L^p}
-\sum_{k=2}^{\infty}2^{{N(\fr{d}{p}-1)}}\big\|\De_{N}U_k\big\|_{L^p}\nonumber\\
&\geq 2^{2N}2^{{N(\fr{d}{p}-3)}}\big\|\De_{N}\big(e^{t\mu \Delta}u_{0,N}-u_{0,N}\big)\big\|_{L^p}-\sum_{k=2}^{\infty}2^{-N}\|U_k\|_{\dot{B}^{d/p}_{p,1}}.
\end{align}
 Obviously, one has
\bbal
e^{t\mu \Delta }u_{0,N}-u_{0,N}=\mu\int_0^t\f(e^{\tau\mu \Delta} \Delta u_{0,N}\g) \dd\tau .
\end{align*}
From which, using Bernstein's inequality (see Lemma \ref{lem2.1}) and Lemma \ref{le1}, we obtain that
\bbal
2^{{N(\fr{d}{p}-3)}}\left\|\De_{N}\big(e^{t\mu\Delta}u_{0,N}-u_{0,N}\big)\right\|_{L^p}\approx\mu\int_0^t\left\|e^{\tau\mu \Delta} u_{0,N}\right\|_{\dot{B}^{d/{p}-1}_{p,\infty}}\dd\tau\les t\|u_{0,N}\|_{\dot{B}^{d/{p}-1}_{p,\infty}}\approx \delta t.
\end{align*}
Combining the above and Lemma \ref{le1} yields that
\bal\label{ylf}
&\quad\mu^{-1}2^{{N(\fr{d}{p}-3)}}\left\|\De_{N}\big(e^{t\mu\Delta}u_{0,N}-u_{0,N}\big)\right\|_{L^p}\nonumber\\
&\geq  t2^{{N(\fr{d}{p}-3)}}\f\|\De_{N}\Delta u_{0,N}\g\|_{L^p}-\int_0^t2^{{N(\fr{d}{p}-3)}}\left\|\De_{N}\Delta\big(e^{\tau\mu\Delta}u_{0,N}-u_{0,N} \big)\right\|_{L^{p}}\dd\tau\nonumber\\
&\geq  Ct2^{{N(\fr{d}{p}-1)}}\f\|\De_{N}u_{0,N}\g\|_{L^p}-C2^{2N}\int_0^t2^{{N(\fr{d}{p}-3)}}\left\|\De_{N}\big(e^{\tau\mu\Delta}u_{0,N}-u_{0,N} \big)\right\|_{L^{p}}\dd\tau\nonumber\\
&\geq C\delta t-C2^{2N}\delta t^2.
\end{align}
Using Proposition \ref{pro1} and \eqref{ylf}, then we deduce from \eqref{lf1} that
\bbal
\|u(t)-u_{0,N}\|_{\dot{B}^{d/{p}-1}_{p,\infty}}
\geq c\delta (t2^{{2N}}-t^22^{4N})-\sum_{k=2}^{\infty} C_0^{k-1} t^{k-1} 2^{2 (k-1) N} \delta^k(1+k)^{-4}.
\end{align*}
Thus, picking $t=\eta2^{-2N}$ with small $\eta$, we deduce that
\bbal
\f\|u(\eta2^{-2N})-u_{0,N}\g\|_{\dot{B}^{d/{p}-1}_{p,\infty}}&\geq c\delta(\eta-C\eta^2)-\delta\sum_{k=1}^{\infty}(C_0\eta \delta)^{k}(2+k)^{-4}\\
&\geq c\delta\f(\eta-C\eta^2-C_0\eta \delta\g)\geq c_0\delta\eta>0.
\end{align*}
This completes the proof of Theorem \ref{th1}.

\section*{Acknowledgments}
Y. Yu is supported by the National Natural Science Foundation of China (12101011). J. Li is supported by the National Natural Science Foundation of China (12161004), Training Program for Academic and Technical Leaders of Major Disciplines in Ganpo Juncai Support Program(20232BCJ23009) and Jiangxi Provincial Natural Science Foundation (20224BAB201008).

\section*{Declarations}
\noindent\textbf{Data Availability} No data was used for the research described in the article.

\vspace*{1em}
\noindent\textbf{Conflict of interest}
The authors declare that they have no conflict of interest.

\end{document}